\documentclass[11pt]{article}
\usepackage{amssymb}
\usepackage{latexsym}
\usepackage{amsthm}
\usepackage{amscd}
\usepackage{amsmath}
\usepackage{indentfirst}
\theoremstyle{definition}
\newtheorem{thm}{Theorem}[section]

\newtheorem{th-def}[thm]{Theorem-Definition}
\newtheorem{cor}[thm]{Corollary}
\newtheorem{defn-lem}[thm]{Definition-Lemma}

\newtheorem{ex}[thm]{Example}
\newtheorem{prop}[thm]{Proposition}

\newtheorem{rem}[thm]{Remark}

\numberwithin{equation}{section}

\allowdisplaybreaks[4]

\def \Q{{\mathbb Q}}

\def \C{{\mathbb C}}
\def \F{{\mathbb F}}

\def \Z{{\mathbb Z}}

\def\map#1.#2.{#1 \longrightarrow #2}
\def\rmap#1.#2.{#1 \dasharrow #2}

\DeclareMathOperator{\Spec}{Spec}
\DeclareMathOperator{\ord}{ord}

\DeclareMathOperator{\Gal}{Gal}

\DeclareMathOperator{\Hom}{Hom}

\DeclareMathOperator{\im}{Im}

\DeclareMathOperator{\Aut}{Aut}

\def\fb#1.{\underset #1 \to \times}
\def\pr#1.{\Bbb P^{#1}}
\def\ring#1.{\mathcal O_{#1}}
\def\mlist#1.#2.{{#1}_1,{#1}_2,\dots,{#1}_{#2}}

\def\Hom{\operatorname{Hom}}

\def\uloopr#1{\ar@'{@+{[0,0]+(-4,5)} @+{[0,0]+(0,10)}
@+{[0,0]+(4,5)}}
  ^{#1}}
\def\dloopr#1{\ar@'{@+{[0,0]+(-4,-5)} @+{[0,0]+(0,-10)}
@+{[0,0]+(4,-5)}}
  _{#1}}

\def\rloopd#1{\ar@'{@+{[0,0]+(5,4)} @+{[0,0]+(10,0)}
@+{[0,0]+(5,-4)}}
  ^{#1}}
\def\lloopd#1{\ar@'{@+{[0,0]+(-5,4)} @+{[0,0]+(-10,0)}
@+{[0,0]+(-5,-4)}}
  _{#1}}

\long\def\ignore#1{}
\long\def\ignore#1{#1}
\title{The representations of the automorphism groups and the Frobenius invariants of K3 surfaces}
\author{}
 \date{}
\begin{document}

\newpage
 \normalsize
\begin{center}
\huge
The representations of the automorphism groups and the Frobenius invariants of K3 surfaces
\end{center}
\vspace{0.43cm}
\begin{center}
\Large
Junmyong Jang
\footnote{This research was supported by Basic Science Research Program through the National Research Foundation of Korea(NRF) funded by the Ministry of Education, Science and Technology(2011-0011428) and KIAS grant funded by the Korea government.}
\end{center}

\vspace{0.2cm}
\begin{center}
\normalsize
email : jmjang@ulsan.ac.kr
\end{center}
\begin{center}
\end{center}
\vspace{0.3cm}
\begin{center}
{\Large
Abstract}
\end{center}
\small
For a complex algebraic K3 surface, it is known that the representations of the automorphism group on the transcendental cycles is finite and is isomorphic
to the representation on the two-forms.
 In this paper we prove similar results for a K3 surface defined over a field of odd characteristic.
 Also we prove that the height and the Artin invariant of a K3 surface equipped with a non-symplectic automorphism of some high order are determined by a congruence class of the base characteristic.

\vspace{0.3cm}
\normalsize
\medskip
     \section{Introduction}
         When $X$ is an algebraic complex K3 surface, the second integral singular cohomology $H^{2}(X,\Z)$ is a free abelian group of rank 22 equipped with a
     lattice structure isomorphic to $U^{3} \oplus E_{8}^{2}$. Here $U$ is the hyperbolic plane and $E_{8}$ is the unique unimodular, even, and  negative definite lattice of rank 8. The cycle map gives a primitive embedding of the Neron-Severi group of $X$ into the second cohomology $NS(X) \hookrightarrow H^{2}(X, \Z)$. The rank of $NS(X)$ is called the Picard number of $X$ and is denoted by $\rho(X)$.
     The orthogonal complement of this embedding is called the transcendental lattice of $X$ and is denoted by $T(X)$. The rank of the transcendental lattice is $22- \rho(X)$. $H^{2}(X, \Z)$ is an overlattice of $NS(X) \oplus T(X)$ and
     $$|H^{2}(X, \Z)/(NS(X) \oplus T(X))| = |d(NS(X))|.$$
     The one dimensional complex space of global holomorphic two-forms of $X$,
     $H^{0}(X,\Omega _{X, \C}^{2})$ is a direct factor of $H^{2}(X, \Z) \otimes \C = H^{2}(X, \C)$ and by the Lefschetz $(1,1)$ theorem,
     $$NS(X) = H^{0}(X,\Omega _{X, \C}^{2})^{\bot} \cap H^{2}(X, \Z)$$
      in $H^{2}(X, \C)$.
      In particular, $H^{0}(X,\Omega _{X/\C}^{2})$ is a direct factor of $T(X) \otimes \C$. The automorphism group of $X$, $\Aut (X)$ has natural actions on  $T(X)$ and on $H^{0}(X,\Omega _{X/\C}^{2})$.
     Let us denote the actions of $\Aut(X)$ on the transcendental lattice and the two-forms by
     \begin{center}
     $\chi _{X} : \Aut(X) \to O(T(X))$ and $\rho _{X} : \Aut(X) \to Gl(H^{0}(X,\Omega _{X/\C}^{2}))$.
     \end{center}
     We say an automorphism of $X$, $\alpha :X \to X$ is symplectic if $\rho _{X}(\alpha)=1$.  If $\alpha$ is of finite order grater than 1
     and the order of $\alpha$ is equal to the order of $\rho_{X}(\alpha)$, we say $\alpha$ is purely non-symplectic.
         Since $H^{0}(X,\Omega _{X/\C}^{2})$ is a direct factor of $T(X) \otimes \C $, there is a canonical surjection
     $p _{X}: \im \chi _{X} \to \im \rho_{X}$.
     It is known that $p_{X}$ is an isomorphism and $\im \chi _{X}$ and $\im \rho _{X}$ are finite cyclic groups,
     \cite{Ni1}. The proof of this result is based on the Lefschetz (1,1) theorem and the Torelli theorem for K3 surfaces.
     If the order of $\im \rho _{X}$ is $N$,
     there is an automorphism $\alpha \in \Aut X$ such that $\xi _{N} =\rho _{X} (\alpha)$ is a primitive $N$-th root of unity.
     Then $T(X)$ has a free $\Z[\xi _{N}]$-module structure in a natural way,
     \cite{MO}, and $22- \rho (X)$ is a multiple of $\phi (N)$. Here $\phi$ is the Euler $\phi$-function. \\

     Assume $k$ is an algebraically closed field of odd characteristic $p$ and $W$ is the ring of Witt vectors of $k$.
     Let $X$ be a K3 surface defined over $k$.
     The second crystalline cohomology $H^{2}_{cris}(X/W)$ and the second \'{e}tale cohomology $H^{2}_{\acute{e}t} (X , \Z _{l})$ are unimodular lattices of rank 22 over $W$ and $\Z _{l}$ respectively. Here $l$ is a prime number which is different from $p$.
     The cycle maps to the crystalline cohomology and \'{e}tale cohomology give
     an embedding of $W$-modules
     $$NS(X) \otimes W \hookrightarrow H^{2}_{cris}(X/W)$$
     and an embedding of $\Z _{l}$-modules
     $$NS(X) \otimes \Z_{l} \hookrightarrow H^{2}_{\acute{e}t}(X,\Z _{l}).$$
     The Newton polygon of $H^{2}_{cris}(X/W)$ is determined by the height of the formal Brauer group of $X$. (See section 2) This height is a positive integer between 1 and 10 or $\infty$. If the height of $X$ is $\infty$, we say $X$ is supersingular.
     We again denote the representation of the automorphism group of $X$ on the global two-forms by
     $$\rho _{X}: \Aut (X) \to Gl (H^{0}(X,\Omega _{X/k}^{2})).$$


     When $X$ is a supersingular K3 surface over $k$, $\rho(X)$ is 22, \cite{M}, \cite{C}, \cite{MA},
     and the discriminant group of the Neron-Severi group
     is $(\Z/p)^{2\sigma}$ for a positive integer $\sigma$ between 1 and 10. We say $\sigma$ is the Artin-invariant of $X$.
      By the Frobenius invariant of a K3 surface in positive characteristic, we mean the height and the Artin-invariant.
     If the Artin-invariant of $X$ is $\sigma$, $H^{2}_{cris}(X/W)/(NS(X)\otimes W)$ is a $\sigma$-dimensional $k$-space and
     there is a canonical projection
     $$H^{2}_{cris}(X/W)/(NS(X) \otimes W) \to H^{2}(X,\mathcal{O}_{X}).$$
          Moreover $H^{2}_{cris}(X/W)/ (NS(X) \otimes W)$
     is an invariant isotropic subspace of the discriminant group $(NS(X)^{*}/NS(X)) \otimes k$. Let
     $$\nu_{X}: \Aut (X) \to O(NS(X)^{*}/NS(X))$$
     be the representations on the discriminant group of the Neron-Severi group.
     We prove that there is a canonical isomorphism $\im \nu _{X} \to \im \rho _{X}$ and $\im \nu _{X} \simeq \im \rho_{X}$
     is a finite cyclic group. (Proposition \ref{pro})\\


     When $X$ is a K3 surface of finite height $h$ over $k$, $\rho(X)$ is at most $22-2h$, \cite{AM}.
     For a K3 surface of finite height $X$,
     we call the orthogonal complements of the cycle maps
     \begin{center}
          $NS(X) \otimes \Z_{l}
     \hookrightarrow H^{2}_{\acute{e}t}(X,\Z_{l})$ and $NS(X) \otimes W \hookrightarrow H^{2}_{cris}(X/W)$
     \end{center}
     the $l$-adic transcendental lattice of $X$ and the crystalline transcendental lattice of $X$ respectively.
     We denote those lattices by $T_{l}(X)$ and $T_{cris}(X)$.
               The representation of $\Aut (X)$ on $T_{l}(X)$ and $T_{cris}(X)$ are denoted by
     \begin{center}
     $\chi _{l,X} : \Aut (X) \to O(T_{l}(X))$ and $\chi _{cris,X} : \Aut (X) \to O(T_{cris}(X))$.
     \end{center}
     We will see $\ker \chi _{l,X}$ is equal to $\ker \chi _{X,cris}$ and
      for any automorphism $\alpha \in \Aut(X)$, the characteristic polynomial of $\chi _{l,X}(\alpha)$ is equal to
      the characteristic polynomial of $\chi _{cris, X}(\alpha)$. (Proposition \ref{pr'})
     And we will construct canonical projections
     \begin{center}
      $p_{cris,X} : \im \chi _{cris,X} \to \im \rho _{X}$ and $p _{l,X} : \im \chi _{l,X} \to \im \rho _{X}$
      \end{center}
     satisfying $\rho _{X} = p_{cris,X} \circ \chi _{cris, X} = p_{l,X} \circ \chi _{l,X}$.
     Also, using a Neron-Severi group preserving lifting of $X$, we prove that $\im \chi_{cris,X}$, $\im \chi _{l,X}$ and $\im \rho _{X}$ are finite.
     (Proposition \ref{pr}, Proposition \ref{pr'}) It follows that, for any $\alpha \in \Aut (X)$, all the eigenvalues of $\chi _{l,X}(\alpha)$ are roots of unity.
          In addition to that, if the order of $\im \chi _{l, X}(\alpha)$ is not divisible by $p$
             and the order of $\rho _{X}(\alpha)$ is $n$, every primitive $n$-th root of unity
     occurs as an eigenvalue of $\chi_{l,X} (\alpha)$. (Proposition 3.7)
     This generalizes Proposition 2.1 in \cite{Ke1}.\\

     When $\alpha$ is an automorphism of a K3 surface $X$ over $k$, under a certain conditions, some parts of eigenvalues of $\alpha ^{*}| H^{2}_{\acute{e}t}(X,\Q_{l})$ are decided by the Frobenius invariant of $X$ and $\rho _{X}(\alpha)$. More precisely we have the following result.\\

       \vspace{0.2cm}
     {THEOREM 3.9.}
      Let $k$ be an algebraically closed field of odd characteristic $p$.
     Assume $X$ is a K3 surface over $k$ and  $\alpha$ is an automorphism of $X$.
     We assume either of the following :\\

     \begin{tabular}{ll}

      &(1) $X$ is of finite height $h$ and the order of $\chi_{l,X}(\alpha)$ is prime to $p$\\

     or & \\

     & (2) $X$ is supersingular of Artin-invariant $\sigma$ and the order of $\alpha$ is finite and prime to $p$.
     \end{tabular}\\

     Suppose $\rho _{X}(\alpha)(u) = \zeta \cdot u$ for a generator $u$ of $H^{0}(X,\Omega _{X/k} ^{2})$
     and $\xi$ is the Teichm\"{u}ller lift of $\zeta$ in $W$.

     Then
     in the case (1),
     $\xi^{\pm p^{0}}, \xi ^{\pm p^{-1}}, \cdots, \xi ^{\pm p^{1-h}}$ appear as eigenvalues of $\chi _{l,X}(\alpha)$
     and in the case (2),
     $\xi^{\pm p^{0}}, \xi ^{\pm p^{-1}} , \cdots , \xi ^{\pm p^{1-\sigma}}$ appear as eigenvalues of $\alpha^{*}|H^{2}_{\acute{e}t}(X, \Q _{l})$.\\

     \vspace{0.2cm}
      Based on Theorem \ref{thm1} and the Tate conjecture for K3 surfaces of finite height, \cite{NO}, \cite{MA}, we can prove the followings.\\

     \vspace{0.2cm}
     {THEOREM 3.10.}
     Let $k$ be an algebraic closure of a finite field of odd characteristic $p$ and $X$ be a K3 surface of finite height $h$ over $k$.
     If the order of $\im \chi_{l,X}$ is not divisible by $p$,
     the projection $p _{l,X} : \im \chi _{X,l} \to \im \rho _{X}$ is bijective.\\

     \vspace{0.2cm}
     {COROLLARY 3.11.}
     Let $k$ be an algebraic closure of a finite field of odd characteristic $p$ and $X$ be a K3 surface of finite height over $k$. If
     $N$ is the order of $\im \rho _{X}$, then the rank of $T_{l}(X)=22-\rho(X)$ is divisible by $\phi (N)$.\\

     \vspace{0.2cm}

      We can apply the above results to study the relation of Frobenius invariant and non-symplectic automorphisms for K3 surfaces.
     We prove the followings.\\

     \vspace{0.2cm}
     {COROLLARY 4.3.}
   Let $k$ be an algebraically closed field of odd characteristic $p$
     and $X$ be a K3 surface over $k$. Let $\alpha$ be an automorphism of $X$.
     We assume that the order of $\rho _{X}(\alpha)$ is $N(>2)$ and that the rank of the Neron-Severi group of $X$
     is at least $22-\phi (N)$.
     If $p^{m} \equiv -1$ modulo $N$ for some $m$, $X$ is supersingular.
     If $p^{m} \not\equiv -1$ modulo $N$ for any $m$ and the order of $p$ in $(\Z /N\Z)^{*}$ is $n$, the height of $X$ is $n$.\\

     \vspace{0.2cm}
     {COROLLARY 4.4.}
      Assume $X$ is a K3 surface over $k$ and $\alpha$ is an automorphism of $X$ such that the order of $\rho _{X}(\alpha)$ is $N(>2)$.
     We assume that $\alpha$ is of finite order prime to $p$ and that a primitive $N$-th root of unity appears only one time in the eigenvalues of $\alpha ^{*}|H^{2}_{\acute{e}t}(X, \Q_{l})$.
     If the order of $p$ in $(\Z /N\Z)^{*}$ is $2n$ and $p^{n} \equiv -1 $ modulo $N$, $X$ is supersingular of Artin-invariant $n$.\\

     \vspace{0.2cm}

     If $X$ is a complex algebraic K3 surface with $N= |\im \rho _{X}|$ and the rank of $T(X)$ is equal to $\phi (N)$, $X$ has a model defined over a number field, \cite{Ri}. By the above results, we deduce that for almost all places, the Frobenius invariant of the reduction of the model of $X$ over the number field is determined by the congruence class of the residue characteristic modulo $N$. (Theorem 4.7.) This generalizes the results on the Delsarte K3 surfaces, \cite{Shi}, \cite{Yu}, \cite{Go}.

         \vspace{0.6cm}
    {\bf Acknowledgment}\\
    The author thanks to Keum, J. and Lee, D. for helpful comments.

     \section{Crystalline cohomology of K3 surfaces}
     In this section, we review some facts on the Neron-Severi group and the crystalline cohomology of K3 surfaces over
     a field of odd characteristic.
     Assume $k$ is an algebraically closed field of characteristic $p>2$.
     Let $W$ be the ring of Witt vectors of $k$ and $K$ be the fraction field of $W$.
     Assume $X$ is a K3 surface over $k$. Let $\widehat{Br}_{X}$ be the formal Brauer group of $X$. $\widehat{Br}_{X}$ is a smooth formal group
     of dimension 1 over $k$, \cite{A}.  A smooth formal group of dimension 1 over an algebraically closed filed of positive characteristic is classified by its height. The height of $\widehat{Br}_{X}$, $h$ is a positive integer $(1 \leq h \leq 10)$ or $\infty$. When $h =\infty$, we say $X$ is supersingular. The Dieudonn\'{e} module of $\widehat{Br}_{X}$ is
     $$\mathbb{D}(\widehat{Br}_{X}) = W[F,V]/(FV=p, F=V^{h-1})$$
     if $h$ is finite or
     $$\mathbb{D}(\widehat{Br}_{X}) = k[[V]]$$
     if $h= \infty$. Here $F$ is a Frobenius linear operator and $V$ is a Frobenius inverse linear operator.\\

     The crystalline cohomologies $H^{i}_{cris}(X/W)$ are finite free $W$-modules of rank $1,0,22,0,1$ for
     $i=0,1,2,3,4$ respectively equipped with Frobenius-linear operators
     $$ \mathbf{F} : H^{i} _{cris}(X/W) \to H^{i} _{cris}(X/W).$$

     If the height $h$ is finite, the Frobenius slopes of $H^{2}_{cris}(X/W)$ are $1-1/h, 1, 1+1/h$ of length $h,22-2h,h$ respectively.
     If $X$ is supersingular, the only Frobenius slope of $H^{2}_{cris}(X/W)$ is 1 of length 22.\\

      The crystalline cohomology $H^{i}_{cris}(X/W)$ can be realized as the hyper-cohomology of the DeRham-Witt complex, \cite{I1},
     $$ 0 \to W\mathcal{O}_{X} \to W\Omega ^{1}_{X/k} \to W\Omega ^{2}_{X/k} \to 0.$$
     The naive filtration of the DeRham-Witt complex gives the slope spectral sequence
     $$H^{i}(X, W\Omega ^{j}_{X/k}) \Rightarrow H^{i+j}_{cris}(X/W).$$
     The $E_{1}$-level page of the slope spectral sequence of $X$ is
     $$\begin{array}{ccccc}
     H^{2}(X,W\mathcal{O}_{X}) & \overset{d}{\to}  & H^{2}(X,W\Omega _{X} ^{1}) & \to & W\\
          0 &  \to & H^{1}(X,W\Omega_{X} ^{1}) & \to & 0 \\
     W & \to & 0 & \to & H^{0}(X, W\Omega _{X} ^{2}).
     \end{array}$$
     Here $H^{2}(X,W\mathcal{O}_{X})$ is isomorphic to $\mathbb{D}(\widehat{Br}_{X})$, \cite{AM}.
     By an exact sequence of sheaves on $X$,
     $$0 \to W\mathcal{O}_{X} \overset{V}{\to} W\mathcal{O}_{X} \to \mathcal{O}_{X} \to 0,$$
     we have an isomorphism $H^{2}(X, W\mathcal{O}_{X})/ V \simeq H^{2}(X, \mathcal{O}_{X})$. \\

      If $X$ is of finite height $h$, $H^{2}(X,W\Omega _{X}^{1})=0$ and the slope spectral sequence degenerates at $E_{1}$-level.
       And $H^{2}_{cris}(X/W)$ has an F-crystal decomposition, \cite{I1}, II.7.2, \cite{Ka}, Theorem 1.6.1,
     $$H^{2}(X/W) = H^{2}_{cris}(X/W)_{[1-1/h]} \oplus H^{2}_{cris}(X/W)_{[1]} \oplus H^{2}_{cris}(X/W)_{[1+1/h]}.$$
     Here
     $$H^{2}_{cris}(X/W)_{[1-1/h]} = H^{2}(X, W\mathcal{O}_{X}) = \mathbb{D}(\widehat{Br}_{X})$$
      and
       $$H^{2}_{cris}(X/W)_{[1+1/h]} =
     \Hom(H^{2}(X,W\mathcal{O}_{X}), H^{4}(X/W)).$$
       Note that $H^{4}(X/W)$ is a free $W$-module of rank 1 equipped with a Frobenius linear operator of slope 2.
       For the cup product pairing on $H^{2}_{cris}(X/W)$, $H^{2}_{cris}(X/W)_{[1-1/h]}$ and $H^{2}_{cris}(X/W) _{[1+1/h]}$ are isotropic and dual to each other.
      On the other hand, $H^{2}_{cris}(X/W)_{[1]}$ is unimodular. The discriminant of the $\Z _{p}$-lattice $H^{2}_{cris}(X/W) _{[1]} ^{F=p}$ is $(-1)^{h+1}$.
     When $X$ is of finite height $h$, the Frobenius morphism and the lattice structure of $H^{2}_{cris}(X/W)$ are completely determined by $h$, \cite{Og2}, p.363.
      Because there exists a canonical embedding, \cite{I1}, Proposition II.5.12,
     $$NS(X) \otimes \Z _{p} \hookrightarrow H^{1}(X,W\Omega^{1} _{X})^{F=p},$$
      the Picard number of $X$, $\rho(X)$ is not greater than the length of slope 1 part of $H^{2}_{cris}(X/W)$. It follows that
     $\rho(X) \leq 22-2h$ if $h$ is finite.\\

       In odd characteristic,
       it is known that $X$ is supersingular if and only if the Picard number of $X$ is 22, \cite{M}, \cite{C}, \cite{MA}. Assume $X$ is a supersingular K3 surface.
      The discriminant of $NS(X)$ is $-p^{2\sigma}$ for an integer $\sigma$ between 1 and 10. $\sigma$ is called the Artin-invariant of $X$.
      The discriminant group of $NS(X)$ is isomorphic to $(\Z / p)^{2\sigma}$.
      Moreover $NS(X)$ is determined by the base characteristic $p$ and $\sigma$, \cite{RS2}.\\

      For a supersingular K3 surface $X$, $H^{0}(X,W\Omega _{X}^{2})=0$ and the slope spectral sequence degenerates at $E_{2}$-level, \cite{I1}, Corollaire II.3.13.
     The only non-trivial map in the $E_{1}$-page of the slope spectral sequence is
     $$d : H^{2}(X,W\mathcal{O}_{X}) \to H^{2}(X,W\Omega _{X}^{1}).$$
     Here $d$ is surjective and
     $$\ker d = H^{2}_{cris}(X/W) / F^{1} H^{2}_{cris}(X/W)$$
     where $F^{\cdot}H^{2}_{cris}(X/W)$ is the filtration given by the slope spectral sequence.
     We can identify $F^{1}H^{2}_{cris}(X/W)$ with the image of the cycle map, \cite{I1}, II.7.2,
     $$NS(X) \otimes W \hookrightarrow H^{2}_{cris}(X/W).$$
     Since $H^{2}_{cris}(X/W)$ is a unimodular $W$-lattice and the cycle map preserves the paring,
     we have a chain
      $$ NS(X) \otimes W \subset H^{2}_{cris}(X/W) \subset (NS(X) ^{*}) \otimes W.$$
      And $\ker d = H^{2}_{cris}(X/W)/(NS(X) \otimes W)$ is an $\sigma$-dimensional isotropic $k$-subspace of the discriminant group $(NS(X) ^{*} \otimes W)/ (NS(X) \otimes W) = (NS(X) ^{*}/NS(X)) \otimes k$.
           It is also known that
     $$\ker dV^{i} : H^{2}(X,W\mathcal{O}_{X}) \to H^{2}(X, W\Omega _{X}^{1})$$
     is a $\sigma -i$ dimensional $k$-space for $i \leq \sigma$
     and $\ker dV^{i+1} \subseteq \ker dV^{i}$, \cite{Ny}, Theorem 0.1. When $x$ is a non-zero element of $\ker dV^{\sigma -1}$,
     $$x, Vx , \cdots , V^{\sigma -i-1}x$$ generate
     $\ker V^{i}d$ over $k$ and
     $x$ is a $V$-adic topological generator of $H^{2}(X, W\mathcal{O}_{X})$.
     The composition
     $$\ker dV^{\sigma - 1} \hookrightarrow H^{2}(X,W\mathcal{O}_{X}) \twoheadrightarrow H^{2}(X,\mathcal{O}_{X})$$
      is an isomorphism.

     \section{Representations of the automorphism groups on the two-forms and transcendental cycles}
     Assume $k$ is an algebraically closed field of characteristic $p>2$. Let $W$ be the ring of Witt vectors of $k$ and $K$ be the fraction field of $W$. Assume $X$ is a K3 surface over $k$.
     Let
     \begin{center}
     $\rho  _{X}: \Aut (X) \to GL(H^{0}(X, \Omega _{X/k}^{2}))$ and $\lambda _{X} : \Aut (X) \to GL(H^{2}(X, \mathcal{O}_{X})) $
     \end{center}
     be the representation of $\Aut(X)$ on
     $H^{0}(X, \Omega _{X/k}^{2})$ and $H^{2}(X, \mathcal{O}_{X})$. By the Serre duality, for any $\alpha \in \Aut X$, $\rho _{X}(\alpha) ^{-1} = \lambda _{X} (\alpha)$ and $\ker \rho _{X} = \ker \lambda _{X}$. \\

     Assume $X$ is supersingular.
     Let
       $$\nu  _{X}: \Aut(X) \to O((NS(X)^{*}/NS(X)) \otimes k)$$
     be the representation of $\Aut(X)$ on the discriminant group $(NS(X)^{*}/NS(X)) \otimes k$.
     Because  $\nu _{X}$ factors through the action of $\Aut(X)$ on $(NS(X)^{*}/NS(X))$, a finite dimensional space over a finite field $\Z /p$,
     $\im \nu _{X}$ is finite. Since $\ker d : H^{2}(X,W\mathcal{O}_{X}) \to H^{2}(X, W\Omega ^{1}_{X})$ is an
     invariant subspace of $(NS(X)^{*}/NS(X)) \otimes k$ for the action of $\Aut X$ and there is a  projection $\ker d \to H^{2}(X,\mathcal{O}_{X})$,
     we have a surjective map $q_{X} : \im \nu _{X} \to \im \lambda _{X}$ such that $q_{X} \circ \nu _{X} = \lambda _{X}$.
     Then a surjective map $p_{X} =q_{X} ^{-1} : \im \nu _{X} \to \im \rho _{X}$ satisfies $p_{X} \circ \nu _{X} = \rho _{X}$.
      \begin{prop}\label{pro}
     Let $X$ be a supersingular K3 surface in odd characteristic. Then $p _{X} : \im \nu _{X} \to \im \rho _{X}$ is an isomorphism.
     \end{prop}
     \begin{proof}
     Let $\sigma$ be the Artin-invariant of $X$ and $x$ be a non-zero element of $\ker dV^{\sigma -1}$.
      Then, $x_{i} = V^{i}x, i=1, \cdots , \sigma -1$ is a basis of $\ker d$.
       Let $y_{i}$ be the dual basis for $x_{i}$ of the dual isotropic subspace of $\ker d$ in $(NS(X)^{*}/NS(X)) \otimes k$.
       Then any automorphism $\alpha \in \Aut (X)$ preserves all the lines
     $k \cdot x_{i}$ and $k \cdot y_{i}$. In other words, all $x_{i}$ and $y_{i}$ are eigenvectors of $\nu _{X}(\alpha )$.
     Since $\alpha ^{*}(V^{i}x) = V^{i}\alpha^{*}(x)$ and $y_{i}$ is dual to $x_{i}$, $\nu _{X} (\alpha)$ is decided by the eigenvalue at $x= x_{0}$. But, for any $\alpha  \in \Aut(X)$, $\rho (\alpha) $ is the inverse of the eigenvalue of $\nu _{X}(\alpha)$ for an eigenvector
     $x_{0}$, so $p_{X}$ is injective.
     \end{proof}

     \begin{rem}
     In \cite{KS}, a supersingular K3 surface is defined to be generic if the order of $\im \nu _{X}$ is 1 or 2. And it is proved that
     there exists a generic supersingular K3 surface of Artin-invariant $\sigma \geq 2$ in odd characteristic, \cite{KS}, Theorem 1.7. By the above proposition,
     a supersingular K3 surface in odd characteristic is generic if and only if the order of $\im \rho _{X}$ is 1 or 2.
      \end{rem}
     For the order of the $\im \rho$, the following is known.
     \begin{prop}(\cite{Ny}, Theorem 2.1)\label{Ny}
     The cardinality of $\im \rho _{X}$ divides $p ^{\sigma} +1$.
     \end{prop}
     \begin{rem}
     If $X$ is a supersingular K3 surface of Artin-invariant 1,  we have shown, using the crystalline Torelli theorem, that
     $\im \rho _{X} \simeq   \im \nu _{X}$ is a cyclic group of rank $p+1$, \cite{J2}, Theorem 3.3. By this result, if $\phi (p+1)>20$,
     $X$ has an automorphism which cannot be lifted to characteristic 0.
     \end{rem}

     Now we assume that $X$ is a K3 surface of finite height over $k$. There is a smooth lifting of $X$ over $W$, $\mathcal{X}/W$, with the generic fiber
     $\mathcal{X}_{K} = \mathcal{X} \otimes K$ such that the reduction map
     $$NS(\mathcal{X}_{K}) \to NS(X)$$
     is an isomorphism, \cite{NO}, \cite{LM}, \cite{J}. We say a lifting of $X$ satisfying this condition is a Neron-Severi group preserving lifting of $X$. When $\mathcal{X}$ is a Neron-Severi group preserving lifting and $\bar{K}$ is an algebraic closure of $K$, the specialization map
     $$\Aut (\mathcal{X}_{K} \otimes \bar{K}) \to \Aut(X)$$
     is an injection of finite index, \cite{LM}, Theorem 6.1.

     \begin{prop}\label{pr}
     Assume $X$ is a K3 surface of finite height over $k$. Then $\im \rho _{X}$ is finite.

     \end{prop}
     \begin{proof}
     Let $\mathcal{X} /W$ be a Neron-Severi group preserving lifting of $X$ and $\mathcal{X}_{K}/K$ be the
     the generic fiber of $\mathcal{X}/W$.
     Since $K$ is of characteristic 0, the image of $\Aut (\mathcal{X}_{\bar{K}}) \to GL(H^{0}(\mathcal{X}_{\bar{K}}, \Omega^{2} _{\mathcal{X}_{\bar{K}}
     /\bar{K}}))$ is finite.
      Therefore the image of $\Aut (\mathcal{X}_{\bar{K}}) \hookrightarrow \Aut(X) \to GL(H^{0}(X, \Omega _{X/k}^{2}))$ is also finite. Since $\Aut(\mathcal{X} _{\bar{K}})$ is of finite index in $\Aut(X)$, the image of $\rho _{X}$ is finite.
     \end{proof}
     Assume $X$ is a K3 surface of finite height over $k$.
     Let $T_{l}(X)$ be the orthogonal complement of the cycle map
         $$NS(X) \otimes \Z_{l} \hookrightarrow H^{2}_{\acute{e}t}(X, \Z_{l})$$
     for $l \neq p$
     and $T_{cris}(X)$ be the orthogonal complement of the cycle map
     $$NS(X) \otimes W \hookrightarrow H^{2}_{cris}(X/W).$$
     We say $T_{l}(X)$ and $T_{cris}(X)$ are the $l$-adic transcendental lattice of $X$ and the crystalline transcendental lattice of $X$ respectively.
     When $\rho(X)$ is the Picard number of $X$, the ranks of $T_{l}(X)$ and $T_{cris}(X)$ are $22-\rho(X)$.
     Note that
     $$H^{2}(X, W\mathcal{O}_{X}) \oplus H^{2}_{cris}(X/W)_{[1+1/h]}$$
     is a direct factor of $T_{cris}(X)$ and
     $H^{2}(X, W\mathcal{O}_{X}) / V \simeq H^{2}(X,\mathcal{O}_{X})$.
     Hence there is a canonical projection
     $T_{cris}(X) \to H^{2}(X,\mathcal{O}_{X})$.
     Let
     \begin{center}
     $\chi _{l,X} : \Aut(X) \to O(T_{l}(X))$ and $\chi _{cris,X} : \Aut(X) \to O(T_{cris}(X))$
     \end{center}
     be the canonical representations.
     \begin{prop}\label{pr'}
     Assume $X$ is a K3 surface of finite height over $k$.
     The images of $\chi _{l,X}$ and $\chi _{cris,X}$ are finite and there is an isomorphism $\psi _{l} : \im \chi _{l,X} \to \im \chi _{cris,X}$ such that
     $\psi _{l} \circ \chi _{l,X} = \chi _{cris,X}$.
     \end{prop}
     \begin{proof}
     Let $\mathcal{X}/W$ be a Neron-Severi group preserving lifting of $X$ with the generic fiber $\mathcal{X}_{K} = \mathcal{X} \otimes K$.
     $\Aut (\mathcal{X}_{\bar{K}})$ is a subgroup of $\Aut(X)$ of finite index.
     If we identify $H^{2}_{\acute{e}t}(\mathcal{X}_{\bar{K}}, \Z_{l})$ with $H^{2}_{\acute{e}t}(X, \Z_{l})$,
     $T_{l}(X)$ is equal to the orthogonal complement of the cycle map
     $$NS(\mathcal{X} _{\bar{K}}) \otimes \Z_{l} \hookrightarrow H^{2}_{\acute{e}t}(\mathcal{X}_{\bar{K}},\Z_{l}).$$
     Because $\bar{K}$ is of characteristic 0, the action of $\Aut (\mathcal{X}_{\bar{K}})$ on $T_{l} (X)$ has a finite image.
     Therefore $\im \chi _{l,X}$ is finite. In a similar way,
     there is a canonical isomorphism
     $$H^{2}_{dr}(\mathcal{X}_{\bar{K}}/\bar{K}) \simeq H^{2}_{cris}(X/W) \otimes \bar{K},$$
     which is compatible with the action of $\Aut (\mathcal{X}_{\bar{K}})$ on both sides, \cite{BO1}, Corollary 2.5.
     Also this isomorphism is compatible with two cycle maps, loc.cit., Corollary 3.7,
     \begin{center}
     $NS(\mathcal{X}_{\bar{K}}) \to H^{2}_{dr}(\mathcal{X}_{\bar{K}}/\bar{K})$ and $NS(X) \to H^{2}_{cris}(X/W) \otimes \bar{K}$.
     \end{center}
     It follows that the action of $\Aut (\mathcal{X}_{\bar{K}})$ on $T_{cris}(X)$ has a finite image and so does the action of $\Aut(X)$ on $T_{cris}(X)$.
              When $\alpha$ is an automorphism of $X$, the characteristic polynomials of $\alpha ^{*}| H^{2}_{cris}(X/K)$ and $\alpha ^{*}| H^{2}_{\acute{e}t}(X,\Q_{l})$ are equal to each other and have integer coefficients, \cite{I0}, 3.7.3.
              Note that the characteristic polynomial of $\alpha ^{*}|H^{2}_{\acute{e}t}(X, \Q_{l})$ is the product of the characteristic polynomial
              of $\alpha ^{*}| NS(X)$ and the characteristic polynomial of $\chi _{l,X}(\alpha)$. Also the characteristic polynomial of
              $\alpha ^{*}| H^{2}_{cris}(X/K)$ is the product of the characteristic polynomial of $\alpha ^{*}| NS(X)$ and the characteristic polynomial of
              $\chi _{cris,X}(\alpha)$.
               Because
         $NS(X)$ is an integral lattice,
         the characteristic polynomial of $\alpha ^{*} | NS(X)$ is also integral and the characteristic polynomials of
      $\chi_{l,X}(\alpha)$ and $\chi _{cris,X}(\alpha)$ are equal to each other and integral.
     Since $\chi_{l,X}(\alpha)$ and $\chi_{cris,X}(\alpha)$ are of finite orders, they are semi-simple and all their eigenvalues are roots of unity.
     It follows that
          $\chi _{l,X}(\alpha)=id$ if and only if $\chi _{cris,X}(\alpha)=id$.
     Therefore $\ker \chi_{l,X}= \ker \chi_{cris,X}$ and there exists a compatible isomorphism
     $\psi_{l} : \im \chi _{l,X} \to \im \chi _{cris,X}.$
      \end{proof}
      Using the projection $T_{cris}(X) \to H^{2}(X, \mathcal{O}_{X})$ and the Serre duality, we have a canonical projection $p_{cirs, X} : \im \chi _{cris,X} \to \im \rho _{X}$
       such that $p _{cris,X} \circ \chi_{cris,X} = \rho _{X}$. Composing with
     $\psi _{l}$, we have a canonical projection $p_{l,X} = p _{cris,X} \circ \psi _{l} : \im \chi _{l,X} \to \im \rho _{X}$.
          \begin{prop}\label{prop}
     Assume $X$ is a K3 surface of finite height and $\alpha$ is an automorphism of $X$. If the order of $\chi_{l,X}(\alpha)$ is prime to $p$ and
     the order of $\rho _{X}(\alpha)$ is $n$,
     all the primitive $n$-th roots of unity appear as eigenvalues of $\chi_{cris,X}(\alpha)$.
     \end{prop}
     \begin{proof}
     Let $\zeta =\rho _{X}(\alpha) \in k^{*}$ and $\xi$ be the Teichm\"{u}ller lifting of $\zeta$ in $W$. Since there is a projection
      $$T_{cris}(X)/p \to H^{2}(X, \mathcal{O}_{X})$$
      $\zeta ^{-1}$ is an eigenvalue of $\alpha ^{*}|( T_{cris}(X)/p)$ and $\xi ^{-1}$ is an
      eigenvalue of $\chi _{cris,X}(\alpha)$. Because the characteristic polynomial of $\chi _{cris,X}(\alpha)$ is integral and $\xi ^{-1}$ is a primitive $n$-th root of unity, the $n$-th cyclotomic polynomial divides the characteristic polynomial of $\chi _{cris, X}(\alpha)$.
      Therefore every primitive $n$-th root of
     unity is an eigenvalue of $\chi _{cris,X}(\alpha)$.
     \end{proof}
     \begin{rem}
     Because the rank of the transcendental lattice is not greater than 21 and the degree of the $n$-th cyclotomic polynomial is $\phi (n)$,
     if an $n$-th root of unity appears as an eigenvalue of
     $\chi _{cris,X}(\alpha)$, then $\phi (n) \leq 21$. Here $\phi$ is the Euler $\phi$-function. In particular, if $p\geq 23$, a $p$-th root of unity can not
     appear as an eigenvalue on $\chi_{cris,X}(\alpha)$.
          \end{rem}
     \begin{thm}\label{thm1}
     Let $k$ be an algebraically closed field of odd characteristic $p$.
     Assume $X$ is a K3 surface over $k$ and  $\alpha$ is an automorphism of $X$.
     We assume either of the following :\\

     \begin{tabular}{ll}

      &(1) $X$ is of finite height $h$ and the order of $\chi_{l,X}(\alpha)$ is prime to $p$\\

     or & \\

     & (2) $X$ is supersingular of Artin-invariant $\sigma$ and the order of $\alpha$ is finite and prime to $p$.
     \end{tabular}\\

     Suppose $\zeta = \rho _{X}(\alpha)$ and $\xi$ is the Teichm\"{u}ller lift of $\zeta$ in $W$.

     Then
     in the case (1),
     $\xi^{\pm p^{0}}, \xi ^{\pm p^{-1}}, \cdots, \xi ^{\pm p^{1-h}}$ appear as eigenvalues of $\chi _{cris,X}(\alpha)$
     and in the case (2),
     $\xi^{\pm p^{0}}, \xi ^{\pm p^{-1}} , \cdots , \xi ^{\pm p^{1-\sigma}}$ appear as eigenvalues of $\alpha^{*}|H^{2}_{cris}(X/W)$.

     \end{thm}

     \begin{proof}
     {\bf First case : X is of finite height}\\
     Assume $X$ is of finite height $h$ and the order of $\chi _{l}(\alpha)$ is prime to $p$.
     Let us identify $H^{2}(X,W\mathcal{O}_{X})$ with $W[F,V]/(FV-p, F-V^{h-1})$.
     Let $f:W \to W$ be the Frobenius morphism.
     We assume
     $$\alpha ^{*}(1) = a_{0}1 + a_{1} V + \cdots + a_{h-1} V^{h-1}.$$
     Here $1 \in W[F,V]/(FV-p,F-V^{h-1})$ is a $V$-adic topological generator and $a_{i} \in W$. Note that $a_{0}$ is a unit of $W$.
     Then
     $$\alpha ^{*}(V^{i}) =V^{i} \alpha ^{*}(1) = f^{-i}(a_{0})V ^{i} + f^{-i}(a_{1}) V^{i+1} + \cdots + f^{-i}(a_{h-1}) V^{h+i-1}$$
     for $i \leq h-1$.
     But $H^{2}(X,W\mathcal{O}_{X})/V = H^{2}(X,\mathcal{O}_{X})$, so $a_{0} \equiv \zeta ^{-1}$ modulo $p$.
     The matrix      of $\alpha ^{*}| (H^{2}(X,W\mathcal{O}_{X})/p)$ with respect to a basis $1+(p), V+(p), \cdots, V^{h-1} + (p)$ is
     $$ \left(
     \begin{array}{cccc}
     \zeta ^{-1} & \cdots & \cdots & \cdots \\
     0 & \zeta ^{-p^{-1}} & \cdots & \cdots \\
     \vdots & \vdots & \ddots & \vdots \\
     0& 0 & \cdots & \zeta ^{-p^{1-h}}
     \end{array}      \right) $$
     and the characteristic polynomial of
     $\alpha ^{*}| (H^{2}(X,W\mathcal{O}_{X})/p)$ is
     $$\prod _{i=0}^{h-1} (T- \zeta ^{- p^{-i}}).$$
     Since $\alpha ^{*}| H^{2}(X, W\mathcal{O}_{X})$ is of finite order prime to $p$, the characteristic polynomial of $\alpha ^{*} | H^{2}(X, W\mathcal{O}_{X})$ is
     $$\prod _{i=0}^{h-1} (T- \xi ^{- p^{-i}}).$$
      Because $H^{2}_{cris}(X/W)_{[1+1/h]}$ is dual to $H^{2}(X,W\mathcal{O}_{X})$, the characteristic polynomial of
     $\alpha ^{*} | H^{2}_{cris}(X/W)_{[1+1/h]}$ is
     $$\prod _{i=0}^{h-1} (T- \xi ^{ p^{-i}}).$$
      Hence the claim follows.\\

      {\bf Second case : X is supersingular}\\
     Assume $X$ is supersingular of Artin-invariant $\sigma$ and $\alpha$ is of finite order prime to $p$.
      Fix $x_{0}$, a non-zero element of
       $$\ker dV^{\sigma -1} : H^{2}(X,W\mathcal{O}_{X}) \to H^{2}(X, W\Omega ^{1}_{X}).$$
       Let $x_{i} = V^{i}x$ for $i=0,1, \cdots , \sigma -1$ and $y_{i}$ is the dual basis of $x_{i}$ in $(NS(X)^{*}/NS(X)) \otimes k$.
      Then $\alpha ^{*} x_{i} = \zeta ^{-p^{-i}} x_{i}$ and $\alpha ^{*} y_{i} =
     \zeta ^{p^{-i}} y_{i}$.
         Because there is an embedding
         $$(NS(X) ^{*} /NS(X)) \otimes k \simeq (pNS(X) ^{*}/ p NS(X)) \otimes k \subseteq (NS(X) \otimes W)/p,$$
        $\xi^{\pm p^{0}}, \xi ^{\pm p^{-1}} , \cdots , \xi ^{\pm p^{1-\sigma}}$ occur as eigenvalues of $\alpha^{*}|NS(X) \otimes W$, so as eigenvalues of $\alpha ^{*}|NS(X)$. Since
        $H^{2}_{cris}(X, W) \otimes K = NS(X) \otimes K$, the claim follows
     \end{proof}

     When $X$ is a complex algebraic K3 surface, the projection $p_{X} : \im \chi_{X}  \to \im \rho _{X}$ is an isomorphism and
     the action of $\Aut (X)$ on the transcendental lattice $T(X)$
     is determined by the action on $H^{0}(X, \Omega ^{2}_{X/\C})$, \cite{Ni1}.
     Moreover if $N$ is the order of $\im \rho _{X}$ and  $\xi _{N}$ is a primitive $N$-th root of unity, by the Lefschtz (1,1) theorem, $T(X)$ is a torsion free $\Z [\xi _{N}]$-module.
     Because $\phi (N) <22$, $\Z[\xi_{N}]$ is a P.I.D,, \cite{MO}, so $T(X)$ is a free $\Z[\xi_{N}]$-module.
     It follows that the rank of $T(X)$ is a multiple of $\phi (N)$.
     We can ask if the same result holds for a K3 surface of finite height in odd characteristic.
     \begin{thm}\label{thm2}
     Let $k$ be an algebraic closure of a finite field of odd characteristic $p$ and $X$ be a K3 surface of finite height $h$ over $k$.
     If the order of $\im \chi_{l,X}$ is not divisible by $p$,
     the projection $p _{l,X} : \im \chi _{l,X} \to \im \rho _{X}$ is bijective.
     \end{thm}
     \begin{proof}
     Clearly $p_{l,X}$ is surjective.
     Suppose $X$ is defined over $\F _{q}$ for $q = p^{m}$.
     The $m$-iterative relative Frobenius morphism of $X/k$ is an endomorphism of $X$ over $k$. We denote this morphism by
     $F : X \to X$. The induced morphism $F^{*}| H^{2}_{\acute{e}t}(X, \Q _{l})$ is equal to the Galois action of the geometric Frobenius
     element in $\Gal(k / \F _{q})$ on $H^{2}_{\acute{e}t}(X, \Q _{l})$.
      Let $V_{l}(X) = T_{l}(X) \otimes \Q _{l}$.
     Then
     $$H^{2}_{\acute{e}t}(X, \Q _{l}) = V_{l}(X) \oplus (NS(X) \otimes \Q_{l}).$$
      Let $\varphi (T)$ be the characteristic polynomial
     of $F^{*}|V_{l}(X)$. $\varphi(T)$ is a polynomial over $\Q$ and equal to
     the characteristic polynomial of $F^{*}|T_{cris}(X)$ \cite{I0}, 3.7.3.
     Let $s_{1}, s_{2} , \cdots , s_{r}$ be the roots of $\varphi(T)$. After replacing $\F_{q}$ by a suitable finite extension,
     we may assume if $s_{i}/s_{j}$ is a root of unity then $s_{i}=s_{j}$.
     Let $\alpha$ be an automorphism of $X$. We may assume $\alpha$ is defined over $\F _{q}$ after replacing the base field $\F _{q}$ by a
     finite extension. In this case, $F \circ \alpha = \alpha \circ F$.
      Since $F^{*}$ and $\alpha ^{*}$ are semi-simple on $V_{l}(X)$, \cite{De}, there exist basis of $V_{l}(X)$ consisting of common eigenvectors for $F^{*}$ and $\alpha ^{*}$.
     We assume $t_{1}, \cdots , t_{r}$ are eigenvalues of $\chi _{l,X}(\alpha)$ and $s_{1}t_{1}, \cdots , s_{r}t_{r}$ are eigenvalues of
     $F^{*} \circ \alpha ^{*}|V_{l}(X)$. Let $\psi(T) \in \Q[T]$ be the characteristic polynomial of $F^{*} \circ \alpha ^{*}|T_{l}(V)$.
       Let us fix an embedding $\bar{\Q} \hookrightarrow \bar{K}$. There is a unique $q$-adic order $\ord _{q}( \cdot )$ on $\bar{\Q}$ associated to
      this embedding.

      Because the height of $X$ is $h$, exactly $h$ roots of $\varphi (T)$ have order $1- 1/h$ for the $q$-adic order $\ord _{q}(\cdot )$.
       Assume $\ord _{q}(s_{i}) = 1-1/h$ for $i=1,\cdots , h$. Then $s_{1}, \cdots , s_{h}$ are roots of characteristic polynomial of
      $F^{*}| H^{2}(X,W\mathcal{O}_{X})$. We assume $\rho _{X} (\alpha)=1$. By the proof of Theorem \ref{thm1}, $\alpha ^{*}| H^{2}(X,W\mathcal{O}_{X})=id$. Because
      the characteristic polynomial of $(F \circ \alpha)^{*} | V_{l}(X)$ is equal to the characteristic polynomial of $(F \circ \alpha)^{*}| T_{cris}(X)$,
      if $\ord _{q}(s_{i})<1$, $t_{i}=1$. Now assume $t_{i} \neq 1$ for some $i>h$. Because the Tate conjecture is valid for K3 surfaces, \cite{NO}, \cite{MA},
      $s_{i}$ is conjugate to $s_{j}$ over $\Q$ for some $j \leq h$. Suppose $\tau (s _{i}) = s _{j}$ for some $\tau \in \Gal (\bar{\Q}/ \Q)$. Then
      $$\tau (s_{i}t_{i})
      = s_{j} \tau(t_{i})= s_{k}t_{k}=s_{k}$$
      for some $k \leq h$. But it is impossible since $\tau(t_{i}) \neq 1$ is a root of unity. Therefore $\chi _{l,X}(\alpha)=id$ and
      $p_{l,X} : \im \chi _{l,X} \to \im \rho _{X}$ is injective.
     \end{proof}

     \begin{cor}\label{cor}
     Let $k$ be an algebraic closure of a finite field of odd characteristic $p$ and $X$ be a K3 surface of finite height over $k$. If
     $N$ is the order of $\im \rho _{X}$, then the rank of $T_{l}(X)$, $22-\rho(X)$ is divisible by $\phi (N)$.

     \end{cor}
     \begin{proof}
     We choose $\alpha \in \Aut(X)$ such that the order of $\rho _{X}(\alpha)$ is $N$.
           Assume that the order of $\chi _{l,X} (\alpha)$ is $p^{m} n$ for some non-negative integer $m$ and $n$ where $p$ does not divides $n$.
      Since the order of $\rho _{X}(\alpha ^{p^{m}})$ is still $N$, replacing $\alpha$ by $\alpha ^{p^{n}}$, we may assume the order of $\chi _{l,X}(\alpha)$ is not divisible by $p$.
      Let $t_{i}$ be an eigenvalue of $\chi _{l,X} (\alpha)$.
     By the proof of Theorem \ref{thm2},
     $t_{i}$ is a primitive $N$-th root of unity
      and $n$ is equal to $N$. It follows that the characteristic polynomial of $\chi_{l,X}(\alpha)$ is a power of $N$-th cyclotomic polynomial over $\Q$ and the rank of $T_{l}(X)$ is a multiple of $\phi (N)$.
     \end{proof}

     \section{Non-symplectic automorphism of some high order and Frobenius invariant}
          \begin{prop}\label{prop2}
     Let $k$ be an algebraic closure of a finite field of odd characteristic $p$
     and $X$ be a K3 surface over $k$. Let $\alpha$ be an automorphism of $X$.
     We assume that the order of $\rho _{X}(\alpha)$ is $N(>2)$  and that the rank of the Neron-Severi group of $X$
     is at least $22-\phi (N)$.
     If $p^{m} \equiv -1$ modulo $N$ for some $m$, $X$ is supersingular.
     If $p^{m} \not\equiv -1$ modulo $N$ for any $m$ and the order of $p$ in $(\Z /N\Z)^{*}$ is $n$, the height of $X$ is $n$.
     \end{prop}
     \begin{proof}
      Assume  $p^{m} \not\equiv -1$ modulo $N$ for any $m$.
      Then by Proposition \ref{Ny},
     $X$ is of finite height.\\

     We assume $X$ is of finite height.
     Then, by the assumption and Corollary 3.11, the rank of $T_{l}(X)$ is $\phi (N)$ and
     the order of $\chi _{l,X}(\alpha)$ is equal to the order of $\rho(\alpha)$.
     Every eigenvalue of $\chi _{l,X}(\alpha)$ is a primitive $N$-th root of unity and
            the characteristic polynomial of $\chi _{l,X}(\alpha)$ is the $N$-th cyclotomic polynomial over $\Q$.
            We denote the $N$-th cyclotomic polynomial over $\Q$ by $\Phi _{N}(T)$.
     Let $\zeta = \rho _{X}(\alpha)$ and  $\xi$ be the Teichm\"{u}ller lift of $\xi$ in $W$.
             Let $V_{l}(X) = T_{l}(X) \otimes \Q _{l}$.
     Since every primitive $N$-th root of unity appears once as an eigenvalue of $\chi _{l,X}(\alpha)$,  $V_{l}(X)$ is a rank 1 free module over $\Q_{l}[T]/ \Phi _{N}(T)$ by the action of $\alpha ^{*}$.
     Note that
     $$\Q _{l}[T]/ \Phi _{N} (T) =(\Q[T]/\Phi _{N}(T)) \otimes \Q_{l}
     \simeq \bigoplus _{k} \Q_{l}(\xi ^{a_{k}})$$
      for suitable primitive $N$-th roots of unity $\xi ^{a _{k}}$.
                  Suppose $X$ and $\alpha$ are defined over $\F _{q}$ for $q =p^{r}$ and $F : X \to X$ is the $r$-iterative relative Frobenius morphism of $X/k$.
                  Let $\varphi(T) \in \Q[T]$ be the characteristic polynomial of $F^{*}|V_{l}(X)$.
                  Since $F^{*} \circ \alpha ^{*} = \alpha ^{*} \circ F^{*}$, $F^{*}| V_{l}(X)$ is a
      $\Q _{l}[T]/\Phi _{N}(T)$-module endomorphism, so it is the multiplication by a unit element of $\Q_{l}[T]/ \Phi _{N}(T)$. Hence all the roots of $\varphi (T)$ are contained in $\Q_{l}(\xi)$ for any $l \neq p$.
         It follows that, by the Chebotarev density theorem, all the roots of $\varphi (T)$ are contained in $\Q (\xi)$.
             Let $n$ be the order of $p$ in $(\Z / N\Z)^{*}$. Then $\xi ^{p^{-n}} = \xi$. Since a primitive $N$-th root of unity appears only one time in the eigenvalues of $\alpha ^{*}| T_{l}(X)$,
     by Theorem \ref{thm1}, the height of $X$ is at most $n$. We fix a $q$-adic order $\ord _{q}(\cdot)$ on $\bar{\Q}$. For a root of $\varphi (T)$, $s \in \Q (\xi)$, by the Tate conjecture, there is $\tau \in \Gal(\bar{\Q} / \Q)$, such that $\ord _{q} (\tau(s)) <1$. Since $\deg \varphi (T)$ is $\phi (N)$
      and the number of primes of $\Q (\xi)$ dividing $p$ is $\phi (N)/n$, the number of roots of $\varphi (T)$ whose $\ord _{q}(\cdot)$ orders are less than 1  is at least $\phi (N) / (\phi (N)/n) = n$. Therefore the height of $X$ is at least $n$, so the height of $X$ is $n$. Now suppose $n=2m$ and
       $p^{m} \equiv -1$ modulo $N$. Then the height of $X$ is $2m$. But among $\xi ^{\pm 1}, \xi ^{\pm p^{-1}}, \cdots , \xi ^{\pm p ^{-2m+1}}$,
      $\xi $ appears twice as an eigenvalue of
     $\chi _{l,X}(\alpha)$. It contradicts to the assumption and $X$ is supersingular.
     \end{proof}

     \begin{rem}
      In the statement of the above theorem,
     the assumption that the rank of Neron-Severi group is at least $22- \phi (N)$ is satisfied if $\phi (N)>10$ by Corollary \ref{cor}.
     \end{rem}

     \begin{cor}\label{cor2}
      Let $k$ be an algebraically closed field of odd characteristic $p$
     and $X$ be a K3 surface over $k$. Let $\alpha$ be an automorphism of $X$.
     We assume that the order of $\rho _{X}(\alpha)$ is $N(>2)$ and that the rank of the Neron-Severi group of $X$
     is at least $22-\phi (N)$.\\

     \begin{tabular}{ll}
     (1) & If $p^{m} \equiv -1$ modulo $N$ for some $m$, $X$ is supersingular.\\

     (2) & If $p^{m} \not\equiv -1$ modulo $N$ for any $m$ and the order of $p$ in $(\Z /N\Z)^{*}$ is $n$,\\
      & the height of $X$ is $n$.
     \end{tabular}
     \end{cor}

     \begin{proof}
     There exists an integral model $\mathcal{X}/R$ of $X/k$, where $R$ is a Noetherian domain of finite type over $\F _{p}$ equipped
     with an embedding $R \hookrightarrow k$ such that
     a geometric generic fiber $k \otimes _{R} \mathcal{X}$ is isomorphic to $X/k$.
     After shrinking the base $\Spec R$, we may assume $NS(X)$ and $\alpha$ extends to
     $\mathcal{X}/R$.
     But the locus of degeneration of the Frobenius invariant is closed, \cite{A}, section 8,
           so we may assume every geometric fiber of $\mathcal{X}/R$ has the same Frobenius invariant as the generic fiber. We choose a closed fiber $X_{0}$ of $\mathcal{X}/R$.
      $X_{0}$ is a K3 surface defined over a finite field. By the assumption,
      the rank of the Neron-Severi group of $X_{0} \otimes \overline{\F} _{p}$ is at least $22-\phi (N)$.
      Then the claim follows by Proposition \ref{prop2}.
      \end{proof}

     \begin{cor}
     Let $k$ be an algebraically closed field of odd characteristic $p$.
     Assume $X$ is a K3 surface over $k$ and $\alpha$ is an automorphism of $X$ such that the order of $\rho _{X}(\alpha)$ is $N(>2)$.
     We assume that $\alpha$ is of finite order prime to $p$ and that a primitive $N$-th root of unity appears only once in the eigenvalues of $\alpha ^{*}|H^{2}_{\acute{e}t}(X, \Q_{l})$.
     If the order of $p$ in $(\Z /N\Z)^{*}$ is $2n$ and $p^{n} \equiv -1 $ modulo $N$, $X$ is supersingular of Artin-invariant $n$.
     \end{cor}
     \begin{proof}
     By the proof of Corollary \ref{cor2}, $X$ is supersingular. Since $n$ is the least number satisfying $p^{n} \equiv -1$ modulo $N$, the Artin-invariant of $X$ is at least $n$ by Proposition \ref{Ny}. On the other hand, by Theorem \ref{thm1}, the Artin-invariant of $X$ can not be bigger than $n$, so it is equal to $n$.
     \end{proof}

     Because a supersingular K3 surface of Artin-invariant 1 is unique up to isomorphism, we obtain the following.
     \begin{cor}
     Assume $k$ is an algebraically closed field of odd characteristic $p$.
     If $10< \phi (N) < 22$, $N \neq 60$ and $p \equiv -1$ modulo $N$, there exists a unique K3 surface over $k$ up to isomorphism which has a purely non-symplectic automorphism of order $N$.
          \end{cor}
     \begin{proof}
     The existence can be checked in Section 3 of \cite{Ke1}.
     \end{proof}

     \begin{rem}
     Over $\C$, a K3 surface equipped with a purely non-symplectic automorphism of some high order is unique, \cite{MO}, \cite{OZ}, \cite{AST}, \cite{Ta}.
     Also there is a unique K3 surface with an automorphism of order 60 in characteristic$\neq 2$ and there is a unique K3 surface with
     an automorphism of order 66 in characteristic$\neq 2,3$, \cite{Ke2}, \cite{Ke3}.
     \end{rem}

       Assume $X$ is a complex algebraic K3 surface such that the order of $\im \chi _{X}$ is $N(>2)$ and
        the rank of the transcendental lattice of $X$ is $\phi (N)$. By \cite{Ri} Corollary 3.9.4, $X$ corresponds to a CM point in a moduli Shimura variety and is defined over a number field.
          We assume $X$, $NS(X)$ and $\Aut (X)$ are defined over a number field $F$ and
        we fix a smooth projective integral model $X_{R}$ of $X$ over a ring $R$, where
        $\Spec R$ is an affine open set of the affine scheme of the ring of integers of $F$,
        $\Spec \mathfrak{o}_{F}$. For each place $\upsilon \in \Spec R$, let $p_{\upsilon}$ be the residue characteristic of $\upsilon$. We may assume $p _{\upsilon} \nmid Nd(NS(X))$ and $p _{\upsilon}$ is unramified in $F$ for any $\upsilon \in \Spec R$.
        We denote the reduction of $X_{R}$
        over an algebraic closure of the residue field $k (\upsilon)$ by $X_{\upsilon}$.

        \begin{thm}
        If $p_{\upsilon} ^{m} \not\equiv -1$ modulo $N$ for all $m \in \Z$, $X_{\upsilon}$ is of finite height and the height of
        $X_{\upsilon}$ is the order of $p_{\upsilon}$ in $(\Z/N)^{*}$.
        If the order of $p_{\upsilon}$ in $(\Z/N)^{*}$ is $2m$ and $p_{\upsilon}^{m} \equiv -1$ modulo $N$,
        $X _{\upsilon}$ is supersingular of Artin-invariant $m$.
         \end{thm}
        \begin{proof}
        There is an embedding
        $$NS(X) \hookrightarrow NS(X _{\upsilon}),$$
        so the rank of $NS(X_{\upsilon})$ is at least $22- \phi (N)$.
        By Corollary \ref{cor2}, $p _{\upsilon}^{m} \not\equiv -1$ modulo $N$ for any $m \in Z$ if and only if $X_{\upsilon}$ is of finite height and in this case,
        the height is equal to the order of $p _{\upsilon}$ in $(\Z/n\Z)^{*}$.\\

        Now assume $X _{\upsilon }$ is supersingular and $2m$ is the order of $p_{\upsilon}$ in $(\Z/ N)^{*}$.
        We fix an automorphism $\alpha \in \Aut(X)$ such that $\xi = \rho _{X}(\alpha)$ is a primitive $N$-th root of unity.
        Note that we do not assume $\alpha$ is of finite order.
        Let $T_{NS}(X)$ be the orthogonal complement of the embedding
        $$NS(X) \otimes W \hookrightarrow NS(X _{\upsilon}) \otimes W.$$
        Here $W$ is the ring of Witt vectors of the algebraic closure of $k(\upsilon)$.
        Because $NS(X_{\upsilon}) \otimes K$
        is canonically isomorphic to $H^{2}_{dr}(X_{R}/R) \otimes K$,
         $\alpha ^{*}| T_{NS}(X)$ is of finite order
        and every $N$-th root of unity appears once as an eigenvalue of $\alpha ^{*}| T_{NS}(X)$.
                Since $p$ does not divide $d(NS(X))$,
        $NS(X) \otimes W$ is unimodular.
        Because there is a unimodular sublattice of $NS(X_{\upsilon}) \otimes W$ of rank $22 - \phi (N)$ ,
        the Artin-invariant of $X_{\upsilon}$ is at most $\phi (N)/2$.
        If $\sigma$ is the Artin-invariant of $X_{\upsilon}$, $N$ divides $p^{\sigma} +1$, so $p ^{\sigma} \equiv -1$ modulo $N$
        and $\sigma$ is an odd multiple of $n$.
        We have an inclusion
                 $$NS(X _{\upsilon}) ^{*}/NS(X_{\upsilon}) \simeq T_{NS}(X) ^{*}/T_{NS}(X) \subseteq T_{NS}(X)/pT_{NS}(X)$$
                 which is compatible with the actions of $\Aut(X)$.
        All the eigenvalues of $\alpha ^{*}| (T_{NS}(X)/pT_{NS}(X))$ are distinct.
        But
        if $\sigma$ is greater than $n$, $\rho _{X_{\upsilon}}(\alpha)$ appears more than once in the eigenvalues of $\alpha ^{*}|(NS(X_{\upsilon})^{*}
        /NS(X_{\upsilon}))$ by Theorem 3.9. This contradicts the assumption.
        Therefore the Artin-invariant of $X_{\upsilon}$ is $n$.
                \end{proof}

     \begin{ex}(c.f. \cite{Shi}, \cite{Yu})
     Assume $X$ is a K3 surface defined over a number field $F$ such that the order of $\im \rho _{X}$ is 36.
     The rank of the transcendental lattice of $X$, $T(X)$ is $12 = \phi (36)$.
     For example, an elliptic K3 surface $X_{36} / \Q$ defined by the equation
     $$ y^{2} = x^{3} + t^{5}(t^{6}-1)$$
     has a purely non-symplectic automorphism of order 36,
     $(t,x,y) \mapsto (\xi ^{30}t, \xi ^{2}x, \xi ^{3} y)$ where $\xi$ is a primitive 36th root of unity.
     Although it is quite believable,
     we do not know whether $X_{36}$ is a unique complex K3 surface satisfying this condition.
     For almost all places $\upsilon$ of $F$, $X$ has a good reduction $X_{\upsilon}$.
     The Frobenius invariant of $X_{\upsilon}$ is following.\\

     \begin{center}
     \begin{tabular}{|c|c|}
     \hline
     congruence class of $p_{\upsilon}$  modulo 36 & Frobenius invariant of $X _{\upsilon}$\\
     \hline
     1 & ordinary \\
     \hline
     17 & height 2 \\
     \hline
     13, 25 & height 3\\
     \hline
     5, 7, 19, 29, 31 & height 6\\
     \hline
     35 & supersingular of Artin-invariant 1\\
     \hline
     11, 23 & supersingular of Artin invariant 3 \\
     \hline
     \end{tabular}
     \end{center}

     \end{ex}

\vskip 1cm

\noindent
J.Jang\\
Department of Mathematics\\
University of Ulsan \\
Daehakro 93, Namgu Ulsan 680-749, Korea\\
Fax : 82-52-259-1692 \\ \\
jmjang@ulsan.ac.kr

     \end{document}